\documentclass[11pt]{article}

\pdfoutput=1
\usepackage{multicol,hyperref}
\usepackage{graphicx,epsfig,wrapfig}
\usepackage{amsmath,amssymb,amsfonts,amsthm,mathrsfs}
\usepackage[usenames,dvipsnames]{color}
\usepackage[ruled]{algorithm2e}

\theoremstyle{plain}
\newtheorem{thm}{Theorem}[section]
\newtheorem{lem}[thm]{Lemma}
\newtheorem{prop}[thm]{Proposition}
\newtheorem{coro}[thm]{Corollary}
\theoremstyle{definition}
\newtheorem{defn}[thm]{Definition}

\theoremstyle{remark}

\newcommand{\probability}[1]{	\mathbb{P}\left\{#1\right\}}
\newcommand{\expectation}[1]{	\mathbb{E}\left[#1\right]}
\newcommand{\variance}[1]{	\mathbf{Var}\left(#1\right)}

\newcommand{\N}{\mathbb{N}}
\newcommand{\R}{\mathbb{R}}

\title{\small\bf Heavy Traffic Limit for a Tandem Queue  with 
Identical Service Times}

\author{ {\small\sc H.\ Christian Gromoll, Bryce Terwilliger, Bert Zwart} \\
{\em\footnotesize University of Virginia and CWI} }

\begin{document}
\maketitle

\begin{abstract}
We consider a two-node tandem queueing network in which the upstream queue is
$M/G/1$ and each job reuses its upstream service requirement when moving
to the downstream queue. Both servers employ the first-in-first-out policy. We
investigate the amount of work in the second queue at certain embedded arrival
time points, namely when the upstream queue has just emptied. We focus on the
case of infinite-variance service times  and obtain a heavy traffic process
limit for the embedded Markov chain.  \end{abstract}



\section{Introduction}

One of the most remarkable queueing models in the literature is the tandem
network consisting of two FIFO (first-in-first-out) queues, where the first
queue is $M/G/1$ with arrival rate $\lambda$, and jobs reuse their original
service requirement when moving to the second queue.  This latter feature
introduces dependence between the second queue's arrival and service
processes, resulting in unusual behavior in the second queue.  In his PhD
thesis \cite{boxmathesis}, Boxma derived explicit expressions for a variety of
steady-state functionals for the second queue, in particular the steady-state
waiting time of a job; see also \cite{boxma1979tandem}.

Apart from being a rare example of a non-product-form tandem queueing network
for which an explicit analysis of the downstream queue is possible, this model
also shows unusual behavior in heavy traffic. In particular, a variety of
results have been derived in the case where the variance of service
requirements is finite; see  \cite{KarpelevichKreinin, kk96} and references
therein for an overview. These results imply that the amount of work in the
second queue is of smaller order than the amount of work in the first queue as
the system load $\rho$ (which is identical for both queues) increases to $1$.
For service times with bounded support, it is even shown in
\cite{boxma1978longest} that the expected value of the waiting time in the
second queue is finite for $\rho=1$.

The intuition behind these results is that the amount of work in the first
queue is driven by sums, but in the second queue is driven by maxima. More
precisely, letting $M_k$ be the largest service time in the $k$th
busy period of an $M/G/1$ queue, and letting $I_k$ be an exponential random
variable with rate $\lambda$ (which can be interpreted as the duration of the
$k$th idle period), the workload $R_k$ in the second queue at the end of the
$k$th busy period of the first queue satisfies the recursion
\begin{equation}
\label{Rchain}
R_{k+1} = \max \{ R_k - I_k, M_k\},\qquad k\geq 1;
\end{equation}
cf.\ \cite{boxma1979tandem}.

The goal of this paper is to analyze the Markov chain $R_k$ in detail, in the
regime where $\rho\rightarrow 1$ and in the situation where normalized sums
and normalized maxima are comparable, i.e.\ the case where service times have
a regularly varying tail of index in the range $(1,2)$. This is the range not
covered in \cite{KarpelevichKreinin, kk96}. We not only focus on the invariant
distribution of this Markov chain, but also on its behavior at the process level.

A key ingredient of our analysis is a limit theorem for the distribution tail
of $M_k$ in heavy traffic. It turns out that it is not possible to use the
tail behavior of $M_k$ for fixed $\rho$, as suggested in \cite{Boxma:2000}.
Rather, we prove a new lower bound for the tail of $M_k$ that is in the same
spirit of an upper bound derived in \cite{boxma1978longest}. A rescaled
version of the distribution of $M_k$ is then shown to converge to a limit that
is expressed through a certain function $\kappa(y)$, shown to be the unique
solution of a particular equation. Once this result for the limiting
distribution of $M_k$ is established, it is possible to utilize techniques
from \cite{ethier2009markov} to determine a Markov (in particular a Feller)
process that is the limit of an appropriately scaled and normalized version of
the Markov chain (\ref{Rchain}).

A model related to (\ref{Rchain}) is treated in \cite{BalleriniResnick}, which
investigates the extreme-value behavior of a Markov chain modelling the evolution of
world records in improving populations.  Though the models are different, one
could connect them by interpreting $R_k-I_k$ as a discounted world record. A
main difference is that in \cite{BalleriniResnick}, the random variables $M_k$
have a fixed distribution, while we need to consider how $M_k$ behaves in
heavy traffic, which represents a substantial part of our effort.

Though the Markov chain (\ref{Rchain}) is of intrinsic interest, it gives a
somewhat coarse description of the workload evolution in the second queue.  It
is also of interest to consider the evolution of the workload in the second
queue during busy periods of the first queue, to consider  joint convergence
of both queues in heavy traffic, and to drop the assumption that inter-arrival
times are exponential. These questions are beyond the scope and techniques of
this paper, and will be pursued elsewhere.

The paper is organized as follows. Section 2 provides a detailed model
description and presents our main results. Section 3 focuses on the behavior
of $M_k$ in heavy traffic. The process limit of (\ref{Rchain}) is investigated
in Section 4 (dealing with convergence of one-dimensional distributions)
and Section 5 (focusing on convergence of the entire process).

\subsection{Notation}
The following notation will be used throughout.  Let $\N=\{1,2,\ldots\}$ and
let $\R$ denote the real numbers.  Let $\R_{+}=[0,\infty)$.  For $a,b\in\R$,
  write $a\vee b$ for the maximum,  $a\wedge b$ for the minimum, $[a]^+=0\vee
  a$, $[a]^-=0\vee -a$ , and $[a]$ for the integer part of $a$. A sum over an
  empty set of indices is defined to be zero.

We say a nonnegative function $f$ is regularly varying with parameter $\nu$ if
\begin{equation*}
\lim_{x\to\infty} f(\lambda x)/f(x)=\lambda^\nu
\end{equation*}
for each $\lambda>0$, and we say it is regularly varying at zero if this holds
for $x\to 0$ instead. A random variable $V$ is regularly varying with
parameter $\nu$ if $x\mapsto \probability{V>x}$ is regularly varying with
parameter $-\nu$.  Note that if a nonnegative random variable $V$ is regularly
varying with parameter $\nu$ then $\expectation{|V|^\gamma}<\infty$ if and
only if $\gamma<\nu$.  For a distribution function $F(x)=\probability{V\leq
x}$ we write $\bar F(x)=1-F(x)$.  Let $\mathbb{D}=\mathbb{D}([0,\infty),\R)$
  be the space of real-valued, right-continuous functions on $[0,\infty)$ with
    finite left limits.  We endow $\mathbb D$ with the Skorohod $J_1$-topology
    which makes $\mathbb D$ a Polish space
    $\cite{billingsley1968convergence}$. If $X$ and $Y$ have the same
    distribution, we write $X\sim Y$. We write $X_n\Rightarrow X$  if  $X_n$
    converges in distribution to $X$.

\section{Model description and main results}
In this section we give a precise description of the tandem queue, specify our assumptions, and state our main result.
\subsection{Definition of the model}\label{sec:modelNor}
We formulate a model equivalent to the one in Boxma \cite{boxma1979tandem}.
The tandem queueing system consists of two queues Q1 and Q2 in series; both Q1
and Q2 are single-server queues employing the FIFO policy, with an unlimited
buffer. Jobs enter the tandem system at Q1. After completion of service at Q1
a job immediately enters Q2, and when service at Q2, which is the exact same
length as previously experienced in Q1, is completed it leaves the tandem
system.  We assume the system is empty at time zero.

Arrivals to Q1 are given by the {\it exogenous arrival process} $E(\cdot)$, a
Poisson process with parameter $\lambda$.  The {\it service times} of
these arriving jobs are given by an i.i.d.\ sequence $\{V_i, i \in \N\}$ with
distribution function $F$. That is, $V_i$ is the  amount of service required
from each server by the $i$th arrival.  We assume throughout that $1-F$
is regularly varying with parameter $-\nu$, $1<\nu<2$, so that
$\expectation{V_1}<\infty$ and $\variance{V_1}=\infty$.

Assume the traffic intensity $\rho=\lambda \expectation{V_1}\leq 1$ so that
the number of jobs in a typical busy period of Q1 is a proper random variable,
and when $\rho<1$ the expected number of jobs in a busy period is
$1/(1-\rho)$.  Let $M_i$ denote the service time of the largest job in the
$i$th busy period of Q1, and denote the distribution function of $M_i$ by $m$.
The distribution function $m$ does not depend on $i$ because the busy periods
correspond to independent and identically distributed cycles.  For $w>0$,
Boxma \cite{boxma1978longest} shows that $m(w)$ is the unique solution to
\begin{equation}\label{eq:Boxma}
m(w)=\int_0^we^{-\lambda t \bar m(w)}dF(t).
\end{equation}

Jobs departing Q1 immediately enter Q2.  Jobs only arrive to Q2 from Q1, so
the arrival process at Q2 is the departure process from Q1.  At Q2, the service
requirement of the $i$th job is $V_i$,  equal to its service requirement at
Q1, so no additional randomness is introduced in the second queue.

For $t\geq 0$, let
\begin{equation}\label{e.idlenessDef}
I(t)=\sup_{s\leq t}\left[\sum_{i=1}^{E(s)}V_i-s\right]^-.
\end{equation}
We interpret $I(t)$ as the cumulative amount of idle time experienced by the
first server up to time $t$.

Let $W_i(t)$ denote the (immediate) workload at time $t$ at Q$i$, $i=1,2$,
which is the total amount of time that the server must work in order to
satisfy the remaining service requirement of each job present at the queue at
time $t$, ignoring future arrivals.  These processes are defined in the usual
way: for $t\geq 0$
\begin{equation}
\nonumber
W_1(t)=\sum_{i=1}^{E(t)}V_i-t +I(t).
\end{equation}
The departure process from Q1 may be written $D(t)=\max\{k\ge0 :\sum_{i=1}^kV_i\le
t-I(t)\}.$ Then $W_2(t)$ is defined analogously to $W_1(t)$ using $D(t)$ in
place of $E(t)$ and the Q2 idleness process in place of $I(t)$; this latter
process is defined as in \eqref{e.idlenessDef} with $D(\cdot)$ in place of
$E(\cdot)$.  

This paper concerns the workload in the second queue at particular points in
time. Let $t_i$ be the arrival time to Q1 of the last job in the $i$th busy
period at Q1. Let $\tilde t_i$ be the time this job arrives to Q2. For
$i\in\N$, $$\tilde t_i =t_i +W_1(t_i).$$ Let $R_n$ be the workload in the
second queue at the time of the arrival to Q2 of the last job in the $n$th
busy period of Q1.  For $n\in \N$, $$R_n=W_2(\tilde t_n).$$ 

The random variable $R_n$ is the largest sojourn time in Q2 experienced by any
job in the $n$th busy period of Q1. The reason for this is that as long as Q1
is not idling, the next interarrival time to Q2 is identical to the next
service requirement, or amount of work to be added to Q2. If this service
requirement is less than the current Q2 workload, the workload will simply
decrease and then increase by the same amount, returning to its previous
level. If this service requirement is greater than the current workload, the
workload will decrease to zero and then jump to a level equal to the incoming
service requirement, higher than the previous level. 

In this way the Q2 workload performs a series of returns to a given level
until a job arrives that is larger than all previous jobs in the busy period,
causing the level to be set higher. Although the last job of a Q1 busy period
may not be the largest, it will by definition return the Q2 workload
to the highest level it attains for the busy period (or set it to a new
highest level if this job happens to be the largest in the busy period). Thus,
the Q2 workload $R_n$ at time $\tilde{t}_n$ is equal to the highest workload
and thus largest sojourn time encountered upon arrival by any job in the $n$th
Q1 busy period. 

The above description is only valid during busy periods of the first queue.
Idleness in the first queue complicates the dynamics substantially.
Nevertheless Boxma \cite{boxma1979tandem} Theorem 6.1 describes the steady
state distribution of $R_n$ when $\rho<1$:
\begin{equation}
\label{steadystate}
\lim_{n\to\infty}\probability{R_n\leq w}=m(w)\exp\left(-\lambda \int_{y=w}^\infty \bar m(y)\,dy\right).
\end{equation}
In this paper we establish a limit theorem for the whole chain
$R_n$ as the traffic intensity $\rho\to 1$. 

\subsection{Heavy traffic limit theorems}
Now we consider a sequence of tandem queueing systems indexed by $n\in\N$.
Each model in the sequence is defined on the same probability space $(\Omega,
\mathcal F, \mathbb P)$.  For each $n\in\N$, the arrival process $E^{(n)}$ is
a Poisson process with parameter $\lambda^{(n)}$, and the service times are
given by the same sequence $\{V_i\}_{i=1}^\infty$ of i.i.d.\ regularly varying
random variables with parameter $\nu\in(1,2)$. Assume that
$\expectation{V_1}>0$ and that $\{V_i\}_{i=1}^\infty$ is independent
of each $E^{(n)}$.  When necessary, we will apply a superscript $(n)$ to
indicate the $n$th model.

{\bf Asymptotic assumptions:} We make the following asymptotic assumptions
about our sequence of models as $n\to\infty$. We want the traffic intensity
$\rho^{(n)}$ increasing to $1$ with fixed service times $\{V_i\}$, so let
$\lambda=1/\expectation{V_1}$ and assume $\lambda^{(n)}\uparrow \lambda$ so that
$\rho^{(n)}=\lambda^{(n)}\expectation{V_1}\uparrow 1$.  Additionally, we
assume this occurs at an appropriate rate, namely 
\begin{equation}\label{e.RateToCritical}
  \left(\frac{1-\rho^{(n)}}{n \bar F(n)}\right)\to
\gamma\geq 0.
\end{equation}

We are now ready to state the first main result of our study. Let $T_\nu$ be a
Pareto$(\nu)$ random variable and let $\Gamma$ denote the gamma function.

\begin{thm}\label{thm:TailMScale}
Under the above assumptions, for $y>0$,
\begin{equation}\label{eq:nmToky}
\lim_{n\to\infty}n\bar m^{(n)}(ny)= \kappa(y)/y,
\end{equation}
where $\kappa=\kappa(y)$ satisfies the equation
\begin{equation} \label{eqinx}
\left(\frac{-1}{\Gamma(1-\nu)}\right)\expectation{e^{-\lambda \kappa T_\nu }}-\kappa \gamma y^{\nu-1} \left(\frac{-1}{\Gamma(1-\nu)}\right)=\left(\lambda \kappa \right)^\nu,
\end{equation}
$\kappa(y)$ is constant when $\gamma=0$, and is regularly varying of index $1-\nu$ when $\gamma > 0$.
\end{thm}

 To give an idea of the proof, observe that Boxma's equation \eqref{eq:Boxma}
 for the distribution function $m$ of the largest job in busy period is nearly
 the Laplace transform of $V$ evaluated at $\lambda \bar m(w)$.  Since
 \eqref{eq:Boxma} holds for each model, we scale time and space by $n$ as in
 the law of large numbers, then apply an Abelian theorem to show that $n\bar
 m^{(n)}(n\cdot)$, the sequence rescaled distribution functions, converges.  We
 then find appropriate asymptotic bounds, establishing subsequential limits.
 These limits can all be characterized as the solution to an equation which is
 shown to be unique, implying convergence. A detailed proof of this result is
 provided in Section 3.\\

We now turn to our results pertaining to the behavior of (\ref{Rchain}) in
heavy traffic.  For each $n\in \N$, let $\{Y_n(k),k=0,1,2,\ldots\}$ be a
Markov chain in $[0,\infty)$ with transition function
  $\mu_n(x,B)=\probability{\max(x-I^{(n)}/n,M^{(n)}/n)\in B}$ where
  $I^{(n)}$ is an exponential random variable with parameter $\lambda^{(n)}$
  independent of the random variable $M^{(n)}$ which is the largest job in a
  busy period.  Observe that, using (\ref{Rchain}) (which is the recursion
  corresponding to proposition \ref{prop:form})
  $Y_n(k)\sim\frac{1}{n}R_k^{(n)}$.  Let $X_n(t)=Y_n([nt])$. Our next result
  describes convergence of the one-dimensional distributions of
  (\ref{Rchain}).

\begin{thm}
\label{thm:one-dimensional-convergence}
For each $t\ge 0$,  $X_n(t)\Rightarrow Z_t$ with
 $$\probability{Z_t\leq x}=\exp\left(-\lambda \int_{y=x}^{x+t/\lambda} \kappa(y)/y \, dy\right).$$
In particular, when $\gamma>0$,
\begin{equation}\label{eq.secondStatement}
\begin{split}
\lim_{t\to \infty}\probability{Z_t\leq x}&=\exp\left(-\lambda\int_x^{\infty}\kappa(y)/y dy \right)\\
&=\lim_{n\to\infty} m^{(n)}(nx)\exp\left(-\lambda^{(n)} \int_{t=x}^\infty n\bar m^{(n)}(nt) dt\right).
\end{split}
\end{equation}
\end{thm}

The second statement \eqref{eq.secondStatement} follows from the first
together with Theorem \eqref{thm:TailMScale}: when
$\rho<1$, we can rescale space by $n$ in the steady state distribution for
\eqref{Rchain} given by \eqref{steadystate}, which becomes
$m^{(n)}(nx)\exp\left\{-\lambda^{(n)} \int_{t=x}^\infty n\bar m^{(n)}(nt)
dt\right\}$.  Note that the limit of the steady state distributions agrees
with the limit of the one dimensional distributions, showing that the limits
$t\rightarrow\infty$ and $n\rightarrow\infty$ can be interchanged. The proof
of the first statement is more involved and described in Section 4.

 We conclude this section with the following theorem for the scaled process.
\begin{thm}\label{thm:Markov}
Suppose $\{X_n(0)\}$ has limiting distribution $\nu$.  There is a Markov process $X$ corresponding to a Feller semigroup $\{T(t)\}$ with initial distribution $\nu$ and sample paths in $D_\mathbb R[0,\infty)$, such that
$X_n\Rightarrow X$.
\end{thm}

The generator of $X$  can informally be written as
\begin{equation}
\hat Af(x)= \frac{-f'(x)}{\lambda} +\int_{y=x}^\infty f'(y)\kappa(y)y^{-1}\,dy.
\end{equation}
A formal proof of Theorem \ref{thm:Markov} is given in Section 5.

\section{The maximum service time in heavy traffic}

The purpose of this section is to derive the asymptotic behavior of the
distribution of $M_k$ in heavy traffic. The section begins with two
technical lemmas, of which the proofs can be skipped at first reading. After
that, we derive asymptotic lower and upper bounds, which are sharp up to a
constant, and provide an important stepping stone towards the derivation of the
limit.

\subsection{Some preliminary lemmas}

The following lemma is intuitive because the supremum over a larger set of
similar objects must also be larger. Recall that
$\lambda^{(n)}\uparrow\lambda$ and let $m^{(\infty)}$ be the distribution
function of the largest job in a busy period in a system where the arrival
process is Poisson and $\rho=1$.

\begin{lem}\label{lem:IncreasingM} As $n\to\infty$,
$$\bar m^{(n)}(x)\uparrow \bar m^{(\infty)}(x),\qquad x\ge0.$$
 \end{lem}
\begin{proof}
Apply \eqref{eq:Boxma} to a convergent subsequence of $\bar{m}^{(n)}(x)$ and
pass to the limit via dominated convergence.  Since \eqref{eq:Boxma} has
unique solutions, the limit must equal $\bar m^{(\infty)}(x)$. For
monotonicity, observe that differentiating \eqref{eq:Boxma} for fixed $x$ with
respect to $\lambda$ yields
\begin{equation*}
\frac{dm}{d\lambda} =\frac{-\int_0^x t(1-m)e^{-\lambda t
(1-m)}\,dF(t)}{1-\lambda \int_0^x te^{-\lambda t(1-m)} dF(t)},
\end{equation*}
which is negative because $\lambda \int_0^x te^{-\lambda t(1-m)}dF(t)\le \lambda
\int_0^x t\, dF(t)\leq\rho$ and implies $\frac{d\bar m}{d\lambda}$ is positive
for $\lambda$ less than the critical value.
\end{proof}
The next lemma uses an Abelian theorem. Recall that $\rho^{(n)}=\lambda^{(n)}
\expectation{V_1}$.  Our assumption that the $\{V_i\}$ are regularly varying
with paramter $1<\nu<2$ implies that we can write $1-F(t)=
\left(\frac{-1}{\Gamma(1-\nu)}\right)t^{-\nu}l(t)$ for a
slowly varying function $l$. 
\begin{lem}\label{lem:RawEquationRM}
 Fix $y>0$.  Then,
\begin{equation*}
\lim_{n\to\infty}\frac{\left(\frac{-1}{\Gamma(1-\nu)}\right)\expectation{e^{-\lambda^{(n)}\bar m^{(n)}(ny)V }\big|V>ny}-\frac{\bar m^{(n)}(ny)(1-\rho^{(n)})(ny)^\nu}{l(ny)}}
{\left(\lambda^{(n)}ny\bar m^{(n)}(ny)\right)^\nu \left(\frac{l\left(\frac{1}{\lambda^{(n)}\bar m^{(n)}(ny)}\right)}{l(ny)}\right)}=1.
\end{equation*}
\end{lem}
\begin{proof}
Since arrivals are Poisson and $\rho^{(n)}\leq 1$ we have
$m^{(n)}(ny)=\int_{0}^{ny} e^{-\lambda^{(n)}t \bar m^{(n)}(ny)}dF(t)$ by
\eqref{eq:Boxma}. So
\begin{equation}\label{eq:BoxmaSum}
\bar m^{(n)}(ny)=1-\int_{0}^\infty e^{-\lambda^{(n)}t \bar m^{(n)}(ny)}dF(t)+\int_{ny}^\infty e^{-\lambda^{(n)}t \bar m^{(n)}(ny)}dF(t).
\end{equation}
For fixed $y>0$, write
\begin{align*}
\int_{ny}^\infty e^{-\lambda^{(n)}t \bar m^{(n)}(ny)}dF(t)&=\int_{0}^\infty e^{-\lambda^{(n)}t \bar m^{(n)}(ny)}1_{(ny,\infty)}(t)dF(t)\\
&=\expectation{e^{-\lambda^{(n)}V\bar m^{(n)}(ny)}1_{(ny,\infty)}(V)}\\
&=\probability{V>ny}\expectation{e^{-\lambda^{(n)}\bar m^{(n)}(ny)V}\Big|V>ny} \\
&=\left(\frac{-1}{\Gamma(1-\nu)}\right)(ny)^{-\nu}l(ny)\expectation{e^{-\lambda^{(n)}\bar m^{(n)}(ny)V}\Big|V>ny}. \\
\end{align*}
Substituting this into equation \eqref{eq:BoxmaSum},
\begin{multline*}
\bar m^{(n)}(ny)=1-\int_{0}^\infty e^{-\lambda^{(n)}t \bar m^{(n)}(ny)}dF(t)\\
+\left(\frac{-1}{\Gamma(1-\nu)}\right)(ny)^{-\nu}l(ny)\expectation{e^{-\lambda^{(n)}\bar m^{(n)}(ny)V}\Big|V>ny}.
\end{multline*}
Rearranging, and using $\lambda^{(n)}\expectation{V}=\rho^{(n)}$ we have
\begin{multline}
\int_{0}^\infty e^{-\lambda^{(n)}t \bar m^{(n)}(ny)}dF(t)-1 +\lambda^{(n)}\bar m^{(n)}(ny)\expectation{V}\\
=\left(\frac{-1}{\Gamma(1-\nu)}\right)(ny)^{-\nu}l(ny)\expectation{e^{-\lambda^{(n)}\bar m^{(n)}(ny)V}\Big|V>ny}-\bar m^{(n)}(ny)(1-\rho^{(n)}).
\end{multline}
Next, dividing by $\left(\lambda^{(n)}\bar m^{(n)}(ny)\right)^\nu l\left(\frac{1}{\lambda^{(n)}\bar m^{(n)}(ny)}\right)$ and multiplying the right hand side by $(ny)^\nu/(ny)^\nu$  we have
\begin{multline}\label{eq:Tauberian}
\frac{\int_{0}^\infty e^{-\lambda^{(n)}t \bar m^{(n)}(ny)}dF(t)-1 +\lambda^{(n)}\bar m^{(n)}(ny)\expectation{V}}{\left(\lambda^{(n)}\bar m^{(n)}(ny)\right)^\nu l\left(\frac{1}{\lambda^{(n)}\bar m^{(n)}(ny)}\right)}
\\=\frac{\left(\frac{-1}{\Gamma(1-\nu)}\right)l(ny)\expectation{e^{-\lambda^{(n)}\bar m^{(n)}(ny)V}\Big|V>ny}-\bar m^{(n)}(ny)(1-\rho^{(n)})(ny)^{\nu}}{\left(\lambda^{(n)}ny\bar m^{(n)}(ny)\right)^\nu l\left(\frac{1}{\lambda^{(n)}\bar m^{(n)}(ny)}\right)}.
\end{multline}

The limit as $n\to \infty$ on the left hand side is 1 by \cite{RegVar} Theorem
8.1.6.  To justify the use of Theorem 8.1.6 we note the left hand side of
equation \eqref{eq:Tauberian} is, in the notation used in Theorem 8.1.6,
$(\hat F(s)-1+s\expectation{V})/(s^\nu l(1/s))$.  So, $\bar F(x)=
-1/\Gamma(1-\nu)x^{-\nu}l(x)$ is equivalent to $(\hat
F(s)-1+s\expectation{V})/(s^\nu l(1/s))\to 1$ where $1<\nu<2$ and
$s=s(n)=\lambda^{(n)}\bar m^{(n)}(ny)$.  Since $\lambda^{(n)}\uparrow
\lambda<\infty$ and $\bar m^{(n)}(\cdot)$ is increasing in $n$ by Lemma
\ref{lem:IncreasingM}, $m^{(\infty)}$ is a proper probability distribution
yields $s\leq \lambda \bar m^{(\infty)}(ny)\downarrow 0$ as $n\to \infty$.
\end{proof}

\subsection{Asymptotic lower and upper bounds}

We are now ready to derive lower and upper bounds for $ny\bar m^{(n)}(ny)$ that are shown to converge in $(0,\infty)$ for each $y>0$.
\begin{lem}\label{limsupfinite}
For all $y\ge0$,
\begin{equation*}
\limsup_{n\to\infty} \,ny\bar m^{(n)}(ny)\leq\max\left[2^{2/\nu}\expectation{V},1\right].
\end{equation*}
\end{lem}
\begin{proof}
If $\lambda^{(n)}(ny)\bar{m}^{(n)}(ny)\geq 1$, we take $A=2$ and $\delta=\nu/2$ in Potter's Theorem \cite{RegVar} 1.5.6 so that for $n$ sufficiently large
\begin{equation*}
(1/2)\left(\lambda^{(n)}ny\bar m^{(n)}(ny)\right)^{-\nu/2}\leq \left(\frac{l\left(\frac{1}{\lambda^{(n)}\bar m^{(n)}(ny)}\right)}{l(ny)}\right).
\end{equation*}
The terms $\bar m^{(n)}(ny)(1-\rho^{(n)})(ny)^\nu$  and $\frac{-1}{\Gamma(1-\nu)}$ are
nonnegative and $l(ny)$ is eventually positive.  So for $n$ sufficiently large,
\begin{equation*}
\begin{split}
&\frac{\left(\frac{-1}{\Gamma(1-\nu)}\right)\expectation{e^{-\lambda^{(n)}\bar m^{(n)}(ny)V }\big|V>ny}-\frac{\bar m^{(n)}(ny)(1-\rho^{(n)})(ny)^\nu}{l(ny)}}
{\left(\lambda^{(n)}ny\bar m^{(n)}(ny)\right)^\nu \left(\frac{l\left(\frac{1}{\lambda^{(n)}\bar m^{(n)}(ny)}\right)}{l(ny)}\right)}\\
&\leq \frac{1}{\left(\lambda^{(n)}ny\bar m^{(n)}(ny)\right)^\nu \left(\frac{l\left(\frac{1}{\lambda^{(n)}\bar m^{(n)}(ny)}\right)}{l(ny)}\right)}\\
&\leq \frac{1}{\left(\lambda^{(n)}ny\bar m^{(n)}(ny)\right)^\nu \left((1/2)\left(\lambda^{(n)}ny\bar m^{(n)}(ny)\right)^{-\nu/2}\right)}\\
&=\frac{2}{\left(\lambda^{(n)}ny\bar m^{(n)}(ny)\right)^{\nu/2}}.
\end{split}
\end{equation*}
Lemma \ref{lem:RawEquationRM} gives
\begin{equation*}
\liminf_{n\to\infty}\frac{2}{\left(\lambda^{(n)}ny\bar m^{(n)}(ny)\right)^{\nu/2}}\geq 1,
\end{equation*}
when $\limsup_{n\to\infty}\lambda^{(n)}ny\bar{m}(ny)\geq 1$.
Since $\lambda^{(n)}\to1/\expectation{V}$, we have
\begin{equation*}
\limsup_{n\to\infty}ny\bar m^{(n)}(ny)\leq \max\left[2^{2/\nu}\expectation{V},1\right].
\end{equation*}
\end{proof}

The following inequality holds even in the case $\rho^{(n)}=1$ for each $n$.
\begin{lem}\label{prop:liminfPositive}
  For each compact set in $\mathcal{K}\subset\mathbb{R}_+$ there exists a constant $L > 0$ such
  that for all $y\in\mathcal{ K}$
\begin{equation*}
\liminf_{n\to\infty} ny\bar m^{(n)}(ny) \geq L.
\end{equation*}
\end{lem}

\begin{proof}
  Fix $y\in\mathcal{K}$ and let $K=\sup_n ny\bar m^{(n)}(ny)$, which is finite
  by Lemma \ref{limsupfinite}.  Note that, for all $t \leq K$ there exists,
  under our assumptions, a constant $C_K$ independent of $n$ such that
  $e^{-\lambda^{(n)} t} \leq 1 - \lambda^{(n)} t + C_K t^2$ for each $n$ and
  each $t\in [0,K]$.  Inserting this inequality, into Boxma's equation
  (\ref{eq:Boxma}) we obtain
\begin{equation*}
\begin{split}
m^{(n)}(ny)&=\int_{0}^{ny} e^{-\lambda^{(n)} t \bar m^{(n)}(ny)}dF(t)\\
&\leq \int_{0}^{ny} \left(1-\lambda^{(n)} t \bar m^{(n)}(ny) +C_K t^2 (\bar m^{(n)} (ny))^2\right) dF(t)\\
&=F(ny) -\lambda^{(n)} \bar m^{(n)}(ny) \int_0^{ny} t dF(t) +C_K \left(\bar m^{(n)}(ny)\right)^2 \int_0^{ny} t^2 dF(t).
\end{split}
\end{equation*}
Consequently,
\begin{equation*}
\bar m^{(n)}(ny)\geq \bar F(ny) +\lambda^{(n)} \bar m^{(n)}(ny) \int_0^{ny} t dF(t) -C_K \left(\bar m^{(n)}(ny)\right)^2 \int_0^{ny} t^2 dF(t).
\end{equation*}
This implies
\begin{equation*}
\bar m^{(n)}(ny) \left( 1-  \lambda^{(n)} \int_0^{ny} t dF(t) + \bar m^{(n)}(ny) C_K \int_0^{ny} t^2 dF(t)\right) \geq \bar F(ny).
\end{equation*}
Since the second factor on the left side is positive, we see that 
\begin{equation}
  \nonumber
\bar m^{(n)}(ny) \geq \bar F(ny) \left(1-  \lambda^{(n)} \int_0^{ny} t dF(t) + \bar m^{(n)}(ny) C_K \int_0^{ny} t^2 dF(t)  \right)^{-1}.
\end{equation}
So we see that 
\begin{equation}\label{eq:boundable}
\frac{1}{ny\bar m^{(n)}(ny)} \leq \frac{1-  \lambda^{(n)} \int_0^{ny} t
dF(t)}{ny \bar F(ny)} +  ny\bar m^{(n)}(ny) C_K \frac{\int_0^{ny} t^2
dF(t)}{(ny)^2 \bar F(ny)}.
\end{equation}
To derive our desired result, we need to show that the limsup on the right
side of this equation is finite.  Since we already know from Lemma
\ref{limsupfinite} that $\limsup_{n\rightarrow\infty} ny\bar m^{(n)}(ny) \leq
K$, it suffices to investigate both fractions. 

Both will be dealt with using Karamata's theorem (Theorems 1.6.4 and 1.6.5 in
\cite{RegVar}). Set $w=ny$. An application of these results in our
setting yields
\begin{equation}
  \label{e.CGadded2}
\lim_{w\to\infty}\frac{\int_{0}^wt^2 dF(t)}{w^2\bar F(w)}=\frac{\nu}{2-\nu}\qquad\text{ and }\qquad \lim_{w\to\infty}\frac{\int_{w}^\infty t dF(t)}{w\bar F(w)}=\frac{\nu}{\nu - 1}.
\end{equation}
For the first fraction, write 
\begin{equation}\label{e.CGadded}
\frac{1-  \lambda^{(n)} \int_0^{w} t dF(t)}{w \bar F(w)} = \frac{1-\rho^{(n)} }{ w \bar F(w)} + \frac{\lambda^{(n)} \int_w^\infty tdF(t) }{w \bar F(w)}.
\end{equation}
The first term on the right side of \eqref{e.CGadded} converges to $\gamma
y^{\nu-1}$ due to our heavy-traffic assumption \eqref{e.RateToCritical} and
since $\bar{F}$ is regularly varying. The second term converges to $\lambda
\nu /(\nu-1)$ by the second equality in \eqref{e.CGadded2}.  Applying the
first equality in \eqref{e.CGadded2} to the second fraction in
\eqref{eq:boundable} yields a limiting upper bound of $\gamma y^{\nu-1}
+\frac{\lambda \nu}{\nu-1} + \frac{K C_K \nu}{2-\nu}.$  So
\begin{equation*}
\liminf_{n\to\infty} ny \bar m^{(n)}(ny)\geq \left(\gamma y^{\nu-1} +\frac{\lambda \nu}{\nu-1} + \frac{K C_K \nu}{2-\nu}\right)^{-1},
\end{equation*}
which is bounded below by some $L>0$ for all $y\in\mathcal{K}$.
\end{proof}

\subsection{Properties of $\kappa$}
In this section we show that $ny\bar m^{(n)}(ny)$ converges to $\kappa(y)$ and we describe several properties of $\kappa(y)$ for fixed $1<\nu<2$, $\lambda>0$, and $\gamma\geq0$.
We begin with several technical lemmas.

\begin{lem}\label{nulessthan2}
If $\lim_{n\to\infty}\bar m^{(n)}(ny)ny=\kappa$ for finite $\kappa$, and $n^{\nu-1}\left(\frac{1-\rho^{(n)}}{l(n)}\right)\to\gamma \left(\frac{-1}{\Gamma(1-\nu)}\right)$ we have
\begin{equation*}
\lim_{n\to\infty}\frac{\bar m^{(n)}(ny)(1-\rho^{(n)})(ny)^\nu}{l(ny)}=\kappa \gamma y^{\nu-1} \left(\frac{-1}{\Gamma(1-\nu)}\right).
\end{equation*}
\end{lem}
\begin{proof}
\begin{equation*}
\begin{split}
\frac{\bar m^{(n)}(ny)(1-\rho^{(n)})(ny)^\nu}{l(ny)}&=\left(\bar m^{(n)}(ny)ny\right)\left(\frac{n^{\nu-1}(1-\rho^{(n)})}{l(n)}\right)\left(\frac{l(n)}{l(ny)}\right)\left(y^{\nu-1}\right)\\
&\to\kappa\gamma \left(\frac{-1}{\Gamma(1-\nu)} \right)y^{\nu-1}.
\end{split}
\end{equation*}
\end{proof}

We will need the following simple fact.

\begin{lem}\label{lem:slowRatio}
Let $f,g:\R_+\to \R_+$ with $f(x)\to\infty$ and $g(x)\to\infty$ as $x\to\infty$ and $f(x)/g(x)\to c>0$ as $x\to\infty$.  Let $L$ be slowly varying.  Then 
\begin{equation*}
\lim_{x\to\infty}\frac{L(f(x))}{L(g(x))}=1.
\end{equation*}
\end{lem}
\begin{proof}
By Karamata's representation theorem we have
\begin{equation*}
\frac{L(f(x))}{L(g(x))}=\frac{\exp\left( \eta(f(x)) + \int_{B}^{f(x)} \frac{\epsilon(t)}{t}\,dt\right)}{\exp\left( \eta(g(x)) + \int_{B}^{g(x)} \frac{\epsilon(t)}{t}\,dt\right)}.
\end{equation*}
Taking the natural log of each side, it suffices to show
\begin{equation*}
\eta(f(x))-\eta(g(x)) + \int_{B}^{f(x)} \frac{\epsilon(t)}{t}\,dt - \int_{B}^{g(x)} \frac{\epsilon(t)}{t}\,dt \to 0,
\end{equation*}
as $x$ goes to infinity.  Since $\eta$ is convergent and $f,g$ go to infinity, we need only show the signed integral
\begin{equation*}
 \int_{g(x)}^{f(x)} \frac{\epsilon(t)}{t}\,dt
 \end{equation*}
 converges to zero.  $\epsilon(t)$ is a bounded positive function, so integrating yields
 \begin{equation*}
\left| \int_{g(x)}^{f(x)} \frac{\epsilon(t)}{t}\,dt\right| \leq \sup_{t\in [g(x)\wedge f(x),g(x)\vee f(x)]}\epsilon(t) |\ln(f(x)/g(x))|.
 \end{equation*}
 Since $g(x)\wedge f(x)$ goes to infinity as $x\to\infty$, and $\epsilon(t)\to 0$ as $t\to\infty$, and $\ln(f(x)/g(x))\to \ln(c)$ as $x\to \infty$.  We have $\ln\left(\frac{L(f(x))}{L(g(x))}\right)\to 0$ as $x\to\infty$.
 \end{proof}
 
\begin{coro}\label{slowratio}
Fix $y>0$.  If $\lim_{n\to\infty}\lambda^{(n)}=\lambda$ and
$\lim_{n\to\infty}\bar m^{(n)}(ny)ny=\kappa$ for $0<\kappa<\infty$, we have
\begin{equation*}
\lim_{n\to\infty} \left(\frac{l\left(\frac{1}{\lambda^{(n)}\bar m^{(n)}(ny)}\right)}{l(ny)}\right)=1.
\end{equation*}
\end{coro}

Recall that $T_\nu$ is a Pareto $\nu$ random variable if
$$\probability{T_\nu>x}=\left\{\begin{array}{ll} x^{-\nu}, & \text{ if } x \geq
  1, \\ 1, & \text{ if } x<1.\end{array}\right.$$
Clearly $T_\nu$ is regularly varying with parameter $\nu$.

\begin{prop}\label{conditionalExpectation} Fix $y>0$. Then if
  $\lambda^{(n)}\to \lambda$, $ny\bar m^{(n)}(ny)\to \kappa>0$, and $V$ is
  regularly varying with parameter $\nu$, we have
\begin{equation*}
\lim_{n\to\infty}\expectation{\left.e^{-\lambda^{(n)}\bar m^{(n)}(ny)V}\right|V>ny}= \expectation{e^{-\lambda \kappa T_{\nu}}}.
\end{equation*}
\end{prop}
\begin{proof}
Observe that
\begin{equation*}
\expectation{\left.e^{-\lambda^{(n)}\bar m^{(n)}(ny)V}\right|V>ny}=\int_{0}^\infty e^{-\lambda^{(n)}\bar m^{(n)}(ny)t} 1_{(ny,\infty)}(t)\, \frac{dF(t)}{1-F(ny)}.
\end{equation*}
Now substitute $u=\bar m^{(n)}(ny)t$ to obtain,
\begin{equation*}
\expectation{\left.e^{-\lambda^{(n)}\bar m^{(n)}(ny)V}\right|V>ny}=\int_{0}^\infty e^{-\lambda^{(n)} u} 1_{(ny\bar m^{(n)}(ny),\infty)}(u)\, \frac{F(du/\bar m^{(n)}(ny))}{1-F(ny)}.
\end{equation*}
thus, we have
\begin{equation*}
\lim_{n\to\infty}\expectation{\left.e^{-\lambda^{(n)}\bar m^{(n)}(ny)V}\right|V>ny}=\lim_{n\to\infty}\int_{0}^\infty e^{-\lambda^{(n)} u} 1_{(\kappa,\infty)}(u)\, \frac{F(du/\bar m^{(n)}(ny))}{1-F(ny)},
\end{equation*}
because $e^{-\lambda^{(n)} u}\leq 1$ and
\begin{equation*}
\begin{split}
\lim_{n\to\infty}\int_{0}^\infty  &\left(1_{(\kappa,\infty)}(u)-1_{(ny\bar m^{(n)}(ny),\infty)}(u)\right)\, \frac{F(\frac{du}{\bar m^{(n)}(ny)})}{1-F(ny)}\\
&=\lim_{n\to\infty}\left( \frac{1-F\left(\frac{\kappa}{\bar m^{(n)}(ny)}\right)}{1-F(ny)} -\frac{1-F\left(\frac{ny\bar m^{(n)}(ny)}{\bar m^{(n)}(ny)}\right)}{1-F(ny)}     \right)\\
&=\lim_{n\to\infty}\left( \frac{1-F\left(\frac{\kappa ny}{ny\bar m^{(n)}(ny)}\right)}{1-F(ny)} -\frac{1-F\left(ny\right)}{1-F(ny)}     \right)\\
&=0,
\end{split}
\end{equation*}
by Lemma \ref{lem:slowRatio} since $1-F$ is regularly varying with parameter $-\nu$.

The measure $ \frac{F(du/\bar m^{(n)}(ny))}{1-F(ny)}$ converges weakly to the
measure $\left(du/\kappa\right)^{-\nu}$ as $n\to\infty$, since for all
$0\le a<b$, as in the previous display, 
\begin{multline*}
\lim_{n\to\infty}\int 1_{(a,b]}\frac{F(du/\bar m^{(n)}(ny))}{1-F(ny)}\\=\lim_{n\to\infty}\frac{F(b/\bar m^{(n)}(ny))}{1-F(ny)}-\frac{F(a/\bar m^{(n)}(ny))}{1-F(ny)}=\left(\frac{a}{\kappa}\right)^{-\nu}-\left(\frac{b}{\kappa}\right)^{-\nu}.
\end{multline*}

For all $\epsilon>0$ there exists $N$ such that $n>N$ implies
$|e^{-\lambda^{(n)} u}-e^{-\lambda u}|<\epsilon$, uniformly in $u$. Combining
with the above weak convergence, 
\[\lim_{n\to\infty} \int_0^\infty \left|e^{-\lambda^{(n)} u}-e^{-\lambda
u}\right|1_{(\kappa,\infty)}(u)\, \frac{F(du/\bar m^{(n)}(ny))}{1-F(ny)}=0.\]
So, we have
\begin{equation*}
\lim_{n\to\infty}\expectation{\left.e^{-\lambda^{(n)}\bar m^{(n)}(ny)V}\right|V>ny}=\lim_{n\to\infty}\int_{0}^\infty e^{-\lambda u} 1_{(\kappa,\infty)}(u)\, \frac{F(du/\bar m^{(n)}(ny))}{1-F(ny)}.
\end{equation*}
Using weak convergence again and the fact that the limit measure has no atoms gives
\begin{equation*}
\lim_{n\to\infty}\expectation{\left.e^{-\lambda^{(n)}\bar m^{(n)}(ny)V}\right|V>ny}= \kappa^\nu\int_\kappa^\infty e^{-\lambda t}\left(dt\right)^{-\nu}.
\end{equation*}
Finally substitute $t=x\kappa$ to get
\begin{equation*}
\kappa^\nu\int_\kappa^\infty e^{-\lambda t}\left(dt\right)^{-\nu}=\int_1^\infty e^{-\lambda \kappa x}\left(dx\right)^{-\nu}=\expectation{e^{-\lambda \kappa T_{\nu}}}.
\end{equation*}
\end{proof}

The equation that describes $\kappa(y)$ is \eqref{eqinx} as shown in the following Lemma.
\begin{lem}\label{lem:eqinKappa}
Let $T_\nu$ be Pareto $\nu$, $1<\nu<2$, $\gamma\geq0$ and $\lambda>0$.  The equation in the variable $\kappa>0$
\begin{equation}
\nonumber
\left(\frac{-1}{\Gamma(1-\nu)}\right)\expectation{e^{-\lambda \kappa T_\nu }}-\kappa \gamma y^{\nu-1} \left(\frac{-1}{\Gamma(1-\nu)}\right)=\left(\lambda \kappa \right)^\nu
\end{equation}
has exactly one solution for all $y>0$.
\end{lem}
\begin{proof}
The left hand side is a strictly decreasing continuous function in $\kappa$
and the right hand side is a strictly increasing continuous function in
$\kappa$.  When $\kappa=0$ the left hand side is $\frac{-1}{\Gamma(1-\nu)}>0$
and the right hand side is 0.  As $\kappa\to \infty$, the left hand side goes
to $0$ if $\gamma=0$ and $-\infty$ if $\gamma>0$; the right hand side goes to
infinity.  Thus  $\eqref{eqinx}$ has exactly one solution.
\end{proof}

We are finally in a position to prove Theorem \ref{thm:TailMScale}.\\

{\it Proof of Theorem \ref{thm:TailMScale}}
Let $\tilde\kappa$ be a limit point of $ny\bar m^{(n)}(ny)$.  Then $0< \tilde
\kappa <\infty$ by Lemmas \ref{prop:liminfPositive} and \ref{limsupfinite}.  Let $n_r$ be a subsequence such that $\lim_{r\to\infty} n_ry\bar m^{(n_r)}(n_ry)=\tilde \kappa$.
By Lemma \ref{lem:RawEquationRM} we have
\begin{equation}\label{eq:subsequenceEquation}
\lim_{r\to\infty}\frac{\left(\frac{-1}{\Gamma(1-\nu)}\right)\expectation{e^{-\lambda^{(n_r)}\bar m^{(n_r)}(n_ry)V }\big|V>n_ry}-\frac{\bar m^{(n_r)}(n_ry)(1-\rho^{(n_r)})(n_ry)^\nu}{l(n_ry)}}
{\left(\lambda^{(n_r)}n_ry\bar m^{(n_r)}(n_ry)\right)^\nu \left(\frac{l\left(\frac{1}{\lambda^{(n_r)}\bar m^{(n_r)}(n_ry)}\right)}{l(n_ry)}\right)}=1.
\end{equation}
Lemmas \ref{conditionalExpectation}, \ref{nulessthan2},  and \ref{slowratio} reduce equation \eqref{eq:subsequenceEquation} to
\begin{equation*}
\frac{\left(\frac{-1}{\Gamma(1-\nu)}\right)\expectation{e^{-\lambda \tilde\kappa T_\nu }}-\tilde \kappa \gamma y^{\nu-1} \left(\frac{-1}{\Gamma(1-\nu)}\right)}{\left(\lambda \tilde\kappa \right)^\nu}=1.
\end{equation*}
Thus, any limit point of $ny\bar m^{(n)}(ny)$ satisfies equation \eqref{eqinx}, of which the solution is called $\kappa (y)$, so Lemma \ref{lem:eqinKappa} implies the limit point is unique, so $\lim_{n\to\infty}ny\bar m^{(n)}(ny)\to \kappa$.  \hfill$\qed$

The properties of $\kappa(y)$ as a function of $y$ are established below.

\subsection{Properties of $\kappa(y)$}

In this section we describe several properties of the function $\kappa$.  In particular $\kappa(y)$ is uniformly bounded above and regularly varying with parameter $1-\nu$.  First we need asymptotic properties of an inverse function.

\begin{lem}\label{prop:InverseRegVar}
Suppose $G:(0,\infty)\to(0,\infty)$ is nonincreasing, invertable, $\lim_{t\downarrow 0}G(t)=\infty$, and $G$ is regularly varying at zero with parameter $-\alpha$ for $0\leq \alpha\leq \infty$.  Then $G^{-1}$ is regularly varying at infinity with parameter ${-1/\alpha}$.
\end{lem}
\begin{proof}
Define $h:\mathbb (0,\infty)\to (0,\infty)$ by $h(t)=1/t$.  We have $G \circ h$ is regularly varying at infinity with parameter $\alpha$, $\lim_{t\to \infty} G \circ h(t)=\infty$, and $G \circ h$ is nondecreasing.  Thus, Proposition 0.8 in \cite{Resnick} gives $(G \circ h)^{-1}=h \circ G^{-1}$ is regularly varying at infinity with parameter $1/\alpha$.  Since the parameter of a composition of regularly varying functions at infinity is the product of the parameters, and $h$ is regularly varying at infinity with parameter $-1$, we have $h \circ h \circ G^{-1}=G^{-1}$ is regularly varying with parameter $-1/\alpha$.
\end{proof}

\begin{lem}\label{lem:propOfKappa}
For fixed $(\lambda, \gamma, \nu)$, $\kappa(y)$ defined implicitly by equation
\eqref{eqinx} is continuous and regularly varying with parameter $1-\nu$ if
$\gamma>0$ and $\kappa(y)$ is constant if $\gamma=0$.  Moreover,
$\kappa(y)\le\frac{1}{\lambda}\left(\frac{-1}{\Gamma(1-\nu)}\right)^{1/\nu}$.
\end{lem}

\begin{proof}
If $\gamma=0$, then $\kappa$ satisfies
$\left(\frac{-1}{\Gamma(1-\nu)}\right)\expectation{e^{-\lambda \kappa
  T_{\nu}}}=(\lambda \kappa)^{\nu}$, so $\kappa$ is constant.
If $\gamma>0$, then $\kappa$ satisfies
\begin{equation*}
\left(\frac{\left(\frac{-1}{\Gamma(1-\nu)}\right)\expectation{e^{-\lambda \kappa T_{\nu}}}-(\lambda \kappa)^{\nu}}{\kappa \gamma\left(\frac{-1}{\Gamma(1-\nu)}\right)}\right)^{1/(\nu-1)}=y.
\end{equation*}
Since
\begin{equation*}
\begin{split}
\kappa&\mapsto \left(\frac{-1}{\Gamma(1-\nu)}\right)\expectation{e^{-\lambda \kappa T_{\nu}}} \text{ is strictly decreasing},\\
\kappa&\mapsto -(\lambda \kappa)^{\nu} \text{ is strictly decreasing and},\\
\kappa&\mapsto \kappa \gamma \left(\frac{-1}{\Gamma(1-\nu)}\right) \text{ is strictly increasing},
\end{split}
\end{equation*}
and each of these functions is continuous, we see that the inverse function $\kappa\mapsto y(\kappa)$ is strictly decreasing and continuous.  So, $y\mapsto\kappa(y)$ is continuous.

The moment generating function of $T_\nu$ is continuous at zero, so 
\[ \kappa \mapsto
  \left(\frac{\left(\frac{-1}{\Gamma(1-\nu)}\right)\expectation{e^{-\lambda
    \kappa T_{\nu}}}-(\lambda \kappa)^{\nu}}{\gamma
      \left(\frac{-1}{\Gamma(1-\nu)}\right)}\right)^{1/(\nu-1)}\]
is a slowly varying function at zero.  Thus, $y(\kappa)$ is regularly varying at zero with parameter $-1/(\nu-1)$.
So, by Lemma \ref{prop:InverseRegVar} we have $\kappa(y)$ is regularly varying at infinity with parameter $1-\nu$.

From equation $\eqref{eqinx}$ we have
\begin{equation*}
\begin{split}
\kappa&=\frac{1}{\lambda}\left(\left(\frac{-1}{\Gamma(1-\nu)}\right)\expectation{e^{-\lambda \kappa T_\nu}}-\kappa \gamma y^{\nu-1}\left(\frac{-1}{\Gamma(1-\nu)}\right)\right)^{1/\nu}\\
&\leq \frac{1}{\lambda}\left(\frac{-1}{\Gamma(1-\nu)}\right)^{1/\nu}.
\end{split}
\end{equation*}
\end{proof}

The following corollary follows from the monotonicity of $\bar m$ and replacing $ny$ in the proof of Theorem \ref{thm:TailMScale} with $ny+b$.  Uniform convergence follows from pointwise convergence and for each $n$, $n\bar m^{(n)}(ny+b)$ is nonincreasing and converging to zero as $y$ goes to infinity, while $\frac{\kappa(y)}{y}$ is continuous and converges to zero. This property also follows from the fact that $\kappa(y)$ is regularly varying, and is necessary in the next section.

\begin{coro}\label{coro:FudgeTail}
Under the assumptions of Lemma \ref{lem:RawEquationRM}, let $b$ be a real
number.  Then
\begin{equation*}
\lim_{n\to\infty}n\bar m^{(n)}(ny+b)=\frac{\kappa(y)}{y}.
\end{equation*}
Moreover, the convergence is uniform on intervals bounded away from zero.
\end{coro}

\section{Convergence of the one-dimensional distributions}

In this section we first write the waiting time in the second queue in terms of independent random variables.  Here we are using the fact that for the M/G/1 queue the length of an idle period is independent of the service times in the preceding busy period.

Recall that $M_k$ is the largest service time the in the $k$th busy period in
the first queue and $I_k$ is the duration of the idle period in the first
queue between the $k$th and $(k+1)$st busy period.

\begin{prop}\label{prop:form}
For each $n\geq 1$.
\begin{equation*}
R_{n}=\max_{k=1}^n \left(M_k-\sum_{j=k}^{n-1}I_j\right),
\end{equation*}
 
\end{prop}

\begin{proof}
This follows by induction in (\ref{Rchain}). 
\end{proof}

We now turn to the distribution of $R_{[nt]}^{(n)}$. We first investigate what
happens if we replace idle periods by their mean, and then show that we can
indeed make such a simplification.

The proof of the following preliminary result is a standard application of weak convergence by considering the sequence of measures $\phi_n=\sum_{k=1}^{[nt]}\frac{1}{n}\delta_{k/n}$.
\begin{prop}\label{prop:Sums}
If $nf(n,ny)\to g(y)$ uniformly on $[0,t]$, and $g$ is continuous on $[0,t]$ then
\begin{equation*}
\lim_{n\to\infty} \sum_{k=1}^{[nt]} f(n,k)=\int_{0}^t g(y)\,dy.
\end{equation*}
\end{prop}

\begin{prop}\label{prop:discountedExtreme}
 Under the assumptions of Lemma \ref{lem:RawEquationRM}, for fixed
 $t\ge 0$ and $x>0$ we have
\begin{multline*}
\lim_{n\to\infty}\probability{\frac{1}{n}\max_{k=1}^{[nt]}
\left(M_k^{(n)}-\frac{k-1}{\lambda}\right)\leq x} \\ =
\left\{\begin{array}{ll}
 \left(1+\frac{t}{x\lambda}\right)^{-\lambda\kappa}, & \text{ if } \gamma=0,\\
\exp\left\{-\lambda\int_x^{x+t/\lambda}\kappa(y)/y\,dy\right\},& \text{ if }
\gamma>0.
\end{array} \right.
\end{multline*}
\end{prop}
\begin{proof}
  Since $\left\{ M_k \right\}$ are independent random variables,
\begin{equation*}
\begin{split}
\probability{\frac{1}{n}\max_{k=1}^{[nt]} \left(M_k^{(n)}-\frac{k-1}{\lambda}\right)\leq x}&=\probability{\max_{k=1}^{[nt]} \left(M_k^{(n)}-\frac{k-1}{\lambda}\right)\leq n x}\\
&=\prod_{k=1}^{[nt]} \probability{M_k^{(n)}-\frac{k-1}{\lambda}\leq n x}\\
&=\prod_{k=1}^{[nt]} \probability{M_k^{(n)}\leq n x+\frac{k-1}{\lambda}}\\
&=\prod_{k=1}^{[nt]} m^{(n)}\left(n x+\frac{k-1}{\lambda}\right)\\
&=\exp\left\{\sum_{k=1}^{[nt]} \ln\left(m^{(n)}\left(n x+\frac{k-1}{\lambda}\right)\right)\right\}.
\end{split}
\end{equation*}
Let $f(n,k)=\ln\left(m^{(n)}\left(nx+\frac{k-1}{\lambda}\right)\right)$. Then
for fixed $y>0$, 
\begin{equation*}
\begin{split}
nf(n,ny)&=n\ln\left(m^{(n)}\left(n x+\frac{ny-1}{\lambda}\right)\right)\\
&=\ln\left(\left(1-\frac{n\bar m^{(n)}\left(n x+ny/\lambda-\frac{1}{\lambda}\right)}{n}\right)^n\right).
\end{split}
\end{equation*}
The function $\ln((1-z/n)^n)\to-z$ uniformly on compact intervals as
$n\to\infty$ and $n\bar
m^{(n)}\left(n\left(x+y/\lambda\right)-\frac{1}{\lambda}\right)\to
\kappa(x+y/\lambda)/(x+y/\lambda)$ uniformly on $y\in[0,t]$ as $n\to\infty$, by
Corollary \ref{coro:FudgeTail} because $x>0$.  Thus, $nf(n,ny)\to
-\kappa(x+y/\lambda)/(x+y/\lambda)$, uniformly for $y\in[0,t]$ since for each $n$, $m^{(n)}$ is nondecreasing and the limit is continuous.
Now continuity of the exponential function and Proposition \ref{prop:Sums} gives
\begin{equation*}
\begin{split}
\lim_{n\to\infty}\probability{\frac{1}{n}\max_{k=1}^{[nt]}
\left(M_k^{(n)}-\frac{k-1}{\lambda}\right)\leq
x}&=\exp\left\{-\int_0^t\frac{\kappa(x+y/\lambda)}{x+y/\lambda}dy\right\} \\
&=\exp\left\{-\lambda\int_x^{x+t/\lambda}\kappa(y)/y\,dy\right\}.
\end{split}
\end{equation*}
Note that the above proof holds for all $\gamma\ge0$; the case $\gamma=0$ is
just a rewriting of the previous expression since $\kappa$ is constant. 
\end{proof}

The sequence of idle periods is i.i.d.\ exponential $\lambda^{(n)}$ in the $n$th system.  Since the largest job in a busy period is independent of the idle period that follows, it is convenient to reindex the sequence of idle periods.  This is why we write $\sum_{i=1}^{k-1}I_i^{(n)}$ instead of $\sum_{i=k}^{n-1}I_i^{(n)}$ in the following proposition.
This proposition shows that, due to monotonicity of the maximum and Kolmogorov's theorem, we can replace $\sum_{i=1}^{k-1}I_i^{(n)}$ by the limit of its mean.

\begin{prop}\label{idlenessNearConstant}
  For all $x>0$, 
\begin{multline}\label{eq:idlenesslim}
\lim_{n\to\infty}\probability{\frac{1}{n}\max_{k=1}^n\left(M_k^{(n)}-\sum_{i=1}^{k-1}I_i^{(n)}\right)\leq x}\\
=\lim_{n\to\infty}\probability{\frac{1}{n}\max_{k=1}^n\left(M_k^{(n)}-\frac{k-1}{\lambda}\right)\leq x}.
\end{multline}
\end{prop}
\begin{proof}
The right hand side of \eqref{eq:idlenesslim} converges by Proposition
\ref{prop:discountedExtreme}. Using the inequality
$-\max(|b_k|) \le \max(a_k +b_k) -\max(a_k) \le \max(b_k) \le \max(|b_k|)$,
which implies that $|\max(a_k +b_k) -\max(a_k)|\le \max(|b_k|)$ for real
numbers $a_k$ and $b_k$, we have 

\begin{multline*}
  \left|\frac{1}{n}\max_{k=1}^n\left(M_k^{(n)}-\sum_{i=1}^{k-1}I_i^{(n)}\right)-\frac{1}{n}\max_{k=1}^n\left(M_k^{(n)}-\frac{k-1}{\lambda}\right)\right| \\
=\Bigg|\frac{1}{n}\max_{k=1}^n\left(M_k^{(n)}-\frac{k-1}{\lambda}+\frac{k-1}{\lambda}-\frac{k-1}{\lambda^{(n)}}+\frac{k-1}{\lambda^{(n)}}-\sum_{i=1}^{k-1}I_i^{(n)}\right)\\
-\frac{1}{n}\max_{k=1}^n\left(M_k^{(n)}-\frac{k-1}{\lambda}\right)\Bigg|\\
\leq\frac{1}{n}\max_{k=1}^n\left|\frac{k-1}{\lambda}-\frac{k-1}{\lambda^{(n)}}+\frac{k-1}{\lambda^{(n)}}-\sum_{i=1}^{k-1}I_i^{(n)}\right|\\
\leq\frac{(n-1)|\lambda-\lambda^{(n)}|}{n\lambda\lambda^{(n)}}+\frac{1}{n}\max_{k=1}^n\left(\left|\frac{k-1}{\lambda^{(n)}}-\sum_{i=1}^{k-1}I_i^{(n)}\right|\right).
\end{multline*}
 So it suffices to show
\begin{equation*}
\frac{1}{n}\max_{k=1}^n\left(\left|\sum_{i=1}^{k-1}\left(I_i^{(n)}-\frac{1}{\lambda^{(n)}}\right)\right|\right)\to 0,
\end{equation*}
in probability as $n\to\infty$.
This follows from Kolmogorov's maximal inequality. For each $\epsilon>0$,
\begin{equation*}
\begin{split}
&\probability{\frac{1}{n}\max_{k=1}^n\left(\left|\sum_{i=1}^{k-1}\left(I_i^{(n)}-\frac{1}{\lambda^{(n)}}\right)\right|\right)\geq\epsilon}\\
 &\qquad=\probability{\frac{1}{n}\max_{k=1}^{n-1}\left(\left|\sum_{i=1}^{k}\left(I_i^{(n)}-\frac{1}{\lambda^{(n)}}\right)\right|\right)\geq\epsilon}\\
&\qquad=\probability{\max_{k=1}^{n-1}\left(\left|\sum_{i=1}^{k}\left(I_i^{(n)}-\frac{1}{\lambda^{(n)}}\right)\right|\right)\geq n\epsilon}\\
&\qquad\leq \frac{1}{(n\epsilon)^2}\frac{n-1}{(\lambda^{(n)})^2}\to 0,
\end{split}
\end{equation*}
as $n\to\infty$ because $\lambda^{(n)}\to\lambda>0$.
\end{proof}

Finally, we prove the main result for this section, which implies Theorem
\ref{thm:one-dimensional-convergence}. Recall that we have assumed
$\rho^{(n)}=\lambda^{(n)} \expectation{V}$ and
$\frac{1-\rho^{(n)}}{n(1-F(n))}\to\gamma\geq 0$ as $n\to\infty$, and that
$1-F(t)= \left(\frac{-1}{\Gamma(1-\nu)}\right)t^{-\nu}l(t)$ for $1<\nu<2$ and
$l$ a slowly varying function.  Recall that $\lambda=\expectation{V}^{-1}$ and
let $\kappa(y)$ be such that the parameters $(\kappa,\lambda, \nu, \gamma, y)$
satisfy \eqref{eqinx}.

\begin{thm}

For fixed $t\ge0$ and $x>0$ we have
\begin{equation*}
\lim_{n\to\infty} \probability{\frac{1}{n}R_{[nt]}^{(n)}\leq x} =
\left\{\begin{array}{ll}
 \left(1+\frac{t}{x\lambda}\right)^{-\lambda\kappa}, & \text{ if } \gamma=0,\\
\exp\left\{-\lambda\int_x^{x+t/\lambda}\kappa(y)/y\,dy\right\},& \text{ if }
\gamma>0.
\end{array} \right.
\end{equation*}
\end{thm}
\begin{proof}
By proposition \ref{prop:form}
\begin{equation*}
\probability{\frac{1}{n}R_{[nt]}^{(n)}\leq x} =\probability{\frac{1}{n}\max_{k=1}^{[nt]} \left(M_k^{(n)}-\sum_{j=k}^{[nt]-1}I_j^{(n)}\right)\leq x}.
\end{equation*}
For each $n$, the iid collections $\{I_k^{(n)}\}$ and $\{M_k^{(n)}\}$ are independent so
\begin{equation*}
\max_{k=1}^{[nt]} \left(M_k^{(n)}-\sum_{j=k}^{[nt]-1}I_j^{(n)}\right)\sim\max_{k=1}^{[nt]} \left(M_k^{(n)}-\sum_{j=1}^{k-1}I_j^{(n)}\right).
\end{equation*}
By proposition \ref{idlenessNearConstant}
\begin{equation*}
\lim_{n\to\infty}\probability{\frac{1}{n}R_{[nt]}^{(n)}\leq x} =\lim_{n\to\infty}\probability{\frac{1}{n}\max_{k=1}^{[nt]}\left(M_k^{(n)}-(k-1)/\lambda\right)\leq x},
\end{equation*}
and so by proposition \ref{prop:discountedExtreme},
\begin{equation}
\label{eq:ZtDist}
\lim_{n\to\infty}\probability{\frac{1}{n}R_{[nt]}^{(n)}\leq x} =
\left\{\begin{array}{ll}
 \left(1+\frac{t}{x\lambda}\right)^{-\lambda\kappa}, & \text{ if } \gamma=0,\\
\exp\left\{-\lambda\int_x^{x+t/\lambda}\kappa(y)/y\,dy\right\},& \text{ if }
\gamma>0.
\end{array} \right.
\end{equation}
\end{proof}

\begin{defn}
Let $\Phi(t,x)$ be the right hand side of \eqref{eq:ZtDist} for $t\geq0$ and $x > 0$, $\Phi(t,x)=0$ for $t\geq 0$ and $x<0$, $\Phi(t,0)=0$ for $t>0$, and $\Phi(0,0)=1$.
\end{defn}

Note that  $x\mapsto\Phi(t,x)$ is a distribution function for each $t\geq 0$.
To see this, recall that  $\kappa(y)$ is constant when $\gamma = 0$ so we may
write
$\Phi(t,x)=\exp\left\{-\lambda\int_x^{x+t/\lambda}\kappa(y)/y\,dy\right\}$ for
all $\gamma\geq0$.  By the proof of Lemma \ref{prop:liminfPositive} we have
\[\kappa(y)\geq \left(\gamma +\frac{\lambda \nu}{\nu-1} + \frac{K C_K
\nu}{2-\nu}\right)^{-1}\]
for $0\leq y \leq 1$, so
$\int_x^{x+t/\lambda}\kappa(y)/y\,dy $ goes to $\infty$ as $x$ goes to zero.
Thus for $t>0$ we have $\Phi(t,x)\downarrow 0$ as $x\downarrow 0$.  Since
$\kappa$ is bounded, for each $t>0$ we have $\Phi(t,x)\to 1$ as $x\to \infty$.
For $t>0$, $\Phi(t,\cdot)$ is strictly increasing because $\kappa(y)/y$ is
strictly decreasing.

\section{Process level convergence}

To prove Theorem \ref{thm:Markov}, we begin with a description of the limit process.

\begin{defn}
Let $Z_t$ be the random variable with distribution function $\Phi(t,x)$,
and observe that $Z_0=0$.
\end{defn}

Then for fixed $t$, $\frac{1}{n}R_{[nt]}^{(n)}$ converges in distribution to $Z_t$.
Note that we have $Z_t\Rightarrow 0$ as $t\downarrow 0$ since $Z_t\geq 0$
a.s.\ and for any $x>0$, $\lim_{t\to0}\probability{Z_t> x}=0,$ for all
$\gamma\ge0$, so $Z_t\to0$ in probability. 

Now we observe a property of $Z_t$ analogous to $N(0,s)+N(0,t)\sim N(0,s+t)$
for $N(0,s)$ and $N(0,t)$  independent normal random variables with mean zero
and variances $s$ and $t$.
\begin{lem}\label{lem:ZtLinear}
Let $Z_t$, $Z_s$, and $Z_{s+t}$ be independent random variables with distribution given by equation \eqref{eq:ZtDist}, then
\begin{equation}
\probability{\max[Z_t-s/\lambda, Z_s]\leq x}=\probability{Z_{s+t}\leq x}.
\end{equation}
\end{lem}

\begin{proof}
Compute
\begin{equation*}
\begin{split}
&\probability{\max[Z_t-s/\lambda, Z_s]\leq x}\\
&\qquad=\probability{Z_t\leq x+s/\lambda}\probability{Z_s\leq x}\\
&\qquad=\exp\left\{{-\lambda \int_{y=x+s/\lambda}^{x+s/\lambda+t/\lambda}\kappa(y)/y\,dy}\right\}\exp\left\{-\lambda \int_{y=x}^{x+s/\lambda}\kappa(y)/y\,dy\right\}\\
&\qquad=\exp\left\{-\lambda \int_{y=x}^{x+(s+t)/\lambda}\kappa(y)/y\,dy\right\}\\
&\qquad=\probability{Z_{s+t}\leq x}.
\end{split}
\end{equation*}
\end{proof}

We are now ready to describe the generator of our limit process.

\begin{defn}
Let $\hat C([0,\infty))$ be the space of continuous functionals on $[0,\infty)$ 
that converge to zero at infinity. The subspace $\hat C_c^1([0,\infty))$ contains
  the continuously differentiable functionals with compact support, and we
  define 
  \[ D = \{ f\in \hat C_c^1([0,\infty)) :\text{for some }a>0,\,\,  |f'(x)|\leq a
    x \text{ for all } x\}.\]
 For each $t\ge 0$ and $f\in\hat C([0,\infty))$, define the operator
\begin{equation}\label{eq:Tf}
T(t)f(x)=\expectation{f(\max(x-t/\lambda, Z_t))}.
\end{equation}
  Note that $T(0)=I$. \end{defn}

A strongly continuous, positive, contraction semigroup on $\hat C([0,\infty))$ whose generator is conservative is called a {\it Feller semigroup}.

\begin{lem}\label{lem:TtSemigroup}
$\{T(t)\}$ defines a Feller semigroup on $\hat C([0,\infty))$ with generator
  $A$, which is an extension to $\hat C([0,\infty))$ of $\hat A$ given by
\begin{equation}
  \hat Af(x)= \frac{-f'(x)}{\lambda} +\int_{x}^\infty
  f'(y)\frac{\kappa(y)}{y}\,dy,\qquad f\in D. 
\end{equation}
 Moreover, $D$ is a core for $A$.
\end{lem}
\begin{proof}
For each $t\geq 0$ and  $f\in \hat C([0,\infty))$, $T(t)f$ is continuous by
  bounded convergence. Since for each $\epsilon>0$ there exists $N$ large
  enough for $x\geq N$ to imply $|f(x)|<\epsilon$, $x\geq N+t/\lambda$ implies
  $\max(x-t/\lambda,Z_t)\geq x-t/\lambda$ and therefore
  $|T(t)f(x)|<\epsilon$. Thus $T(t):\hat C([0,\infty))\to \hat C([0,\infty))$.
    Clearly $T(t)$ is also positive, linear, and contractive on $\hat
    C([0,\infty))$.  
      
For the semigriup property, let $Z_s$, $Z_t$, and $Z_{s+t}$ be independent
with distributions given by
\eqref{eq:ZtDist}. Then by Fubini's theorem and Lemma \ref{lem:ZtLinear},
\begin{equation}
\begin{split}
T(s)T(t)f(x)&=\expectation{T(t)f(\max(x-s/\lambda, Z_s))}\\
&=\expectation{f(\max(\max(x-t/\lambda, Z_t)-s/\lambda, Z_s))}\\
&=\expectation{f(\max(x-(s+t)/\lambda, Z_t-s/\lambda, Z_s))}\\
&=\expectation{f(\max(x-(s+t)/\lambda, Z_{s+t}))}\\
&=T(s+t)f(x)
\end{split}
\end{equation}
for all $f\in \hat C([0,\infty))$, $x\in [0,\infty)$, and $s,t\geq 0$.  Since $T(0)=I$, this implies $\{T(t)\}$ is a semigroup.

Next we show the semigroup is strongly continuous. 
Note that $f\in  \hat C([0,\infty))$ is uniformly continuous. Write
\begin{equation}
\begin{split}
\| T(t)f-f\| &= \sup_{x\in[0,\infty)}\Bigl|\expectation{f(\max(x-t/\lambda,
Z_t))-f(x)}\Bigr|\\
&\leq \sup_{x\in[0,\infty)}\probability{Z_t\leq x-t/\lambda}\left|f(x-t/\lambda)-f(x)\right|\\
  &\qquad +\sup_{x\in[0,\infty)}\int_{[x-t/\lambda]^+}^\infty\left|f( y)-f(x)\right|\Phi(t,dy).
\end{split}
\end{equation}

As $t\downarrow 0$ the first term goes to zero because it is bounded by
\[\sup_{x\in[t/\lambda,\infty)}\left|f(x-t/\lambda)-f(x)\right|\to 0\]
by uniform continuity of $f$.
For the second term let $\epsilon>0$ and $\eta>0$ such that $\sup_{x,y\in [0,\eta)}|f(x)-f(y)|<\epsilon$. Then
\begin{multline}
\sup_{x\in[0,\infty)}\int_{[x-t/\lambda]^+}^\infty\left|f(
y)-f(x)\right|\Phi(t,dy) \leq\sup_{x\in[\eta,\infty)}\probability{Z_t\ge\eta-t/\lambda}2\left\|f\right\|\\
 \qquad \vee
 \left(\sup_{x\in[0,\eta)}\sup_{y\in[\eta,\infty)}\probability{Z_t\geq\eta}\left|f(y)-f(x)\right|\right.
   \\
 \left. + \sup_{x,y\in[0,\eta)}\probability{Z_t<\eta}\left|f( y)-f(x)\right|\right).
\end{multline}
We have $\displaystyle
\sup_{x\in[\eta,\infty)}\probability{Z_t\ge\eta-t/\lambda}\left\|f\right\|\to
  0$ since $f$ is bounded, $Z_t\Rightarrow 0$, and $\eta>0$. Similarly,
\[
\sup_{x\in[0,\eta)}\sup_{y\in[\eta,\infty)}\probability{Z_t\geq\eta}\left|f(y)-f(x)\right|\leq
  \probability{Z_t\geq\eta}2\|f\| \to 0\]
  since $f$ is
    bounded and $Z_t\Rightarrow 0$. Lastly,
$\displaystyle \sup_{x,y\in[0,\eta)}\probability{Z_t<\eta}\left|f(
  y)-f(x)\right|<\epsilon$ and since $\epsilon$ is arbitrary, we conclude that
  $\| T(t)f-f\| \to 0$ as $t\downarrow 0$. Thus $\{T(t)\}$ is a strongly
  continuous semigroup.  $A$ is conservative because $(f\equiv 1,f\equiv 0)$
  is in the bounded-pointwise closure of $A$.  

  It remains to show that an extension $A$ of $\hat{A}$ is the generator of
  $\{T(t)\}$ and that $D$ is a core for $A$. 
Note that by l'H\^opital's rule, for $w>0$,
\begin{equation*}
\begin{split}
\lim_{t\to 0}\frac{1-\Phi(t,w)}{t}&=\lim_{t\to 0}-\frac{\partial}{\partial t}\Phi(t,w)\\
&=\lim_{t\to 0}\frac{\kappa(w+t/\lambda)}{w+t/\lambda}\Phi(t,w)\\
&=\kappa(w)/w,
\end{split}
\end{equation*}
since $\kappa$ is continuous by Lemma \ref{lem:propOfKappa}.
Now let $f\in D$. Then
\begin{equation}
\begin{split}
\frac{1}{t}\left(T(t)f(x)-f(x)\right)&=\expectation{\frac{1}{t}\left(f(\max(x-t/\lambda, Z_t))-f(x)\right)}\\
&=\probability{Z_t\leq x-t/\lambda}\left(\frac{f([x-t/\lambda]^+)-f(x)}{t}\right)\\
&\qquad+\frac{1}{t}\int_{[x-t/\lambda]^+}^\infty \left(f(y)-f(x)\right)\,\Phi(t,dy).\\
\end{split}
\end{equation}
As $t\to 0$, the first term converges to $-f'(x)/\lambda$  if $x>0$ and zero if $x=0$ since
$f$ is differentiable and $Z_t\Rightarrow 0$.

For the second term we note the integral is finite since $f$ is bounded and $\Phi(t,dy)$ is a probability measure, so Fubini's theorem yields
\begin{equation}
\begin{split}
\frac{1}{t}\int_{[x-t/\lambda]^+}^\infty &(f(y)-f(x))\,\Phi(t,dy)\\
&=\frac{1}{t}\int_{[x-t/\lambda]^+}^\infty\int_{x}^y f'(w)\,dw\,\Phi(t,dy)\\
&=\frac{1}{t}\int_{[x-t/\lambda]^+}^x\int_{[x-t/\lambda]^+}^w -f'(w)\,\Phi(t,dy)dw\\
&\qquad +\frac{1}{t}\int_{x}^\infty\int_{w}^\infty f'(w)\,\Phi(t,dy)dw\\
&=\frac{1}{t}\int_{[x-t/\lambda]^+}^x -f'(w)\left(\Phi(t,w)-\Phi(t,[x-t/\lambda]^+)\right)dw\\
&\qquad  +\int_{x}^\infty f'(w)\left(\frac{1-\Phi(t,w)}{t}\right)dw.
\end{split}
\end{equation}
The first integral converges to zero because $f'$ is bounded and $\Phi$ is
continuous; the region shrinks to zero at a rate proportional to $t$.  For
$x>0$, the second integral converges to $\int_{x}^\infty f'(w)\kappa(w)/w\,dw$
because $\left(\frac{1-\Phi(t,w)}{t}\right)\to \kappa(w)/w$ uniformly since
$\Phi(t,\cdot)$ is increasing and continuous for each $t$.  If $x = 0$ then
for any $\delta>0$ we
split the integral into the part over $[\delta,\infty]$, on which we may use
the above uniform convergence, and must then show that 
\begin{equation}\label{e.zeroint}
  \int_0^\delta f'(w)\frac{1-\Phi(t,w)}{t}dw\to 0,\qquad \text{ as
  }\delta\to0,
\end{equation}
uniformly in $t$. To that end, we use the bound $1-\exp\{-\lambda z\}\le
\lambda z$ to write  
\begin{align*}
  1-\Phi(t,w)& = 1-\exp\left( -\lambda\int_w^{w+t/\lambda}\frac{\kappa(y)}{y}dy
  \right)\\
  & \le \lambda\frac{\|\kappa\|_\infty t}{\lambda w}. 
\end{align*}
Combining with the bound  $|f'(w)|\le aw$, the integral in \eqref{e.zeroint}
is bounded above by $\delta a\|\kappa\|_\infty$, achieving the desired result.  

We have shown the generator of $\{T(t)\}$ extends $\hat A$. To show $D$ is a
core, note first that for all $t\geq0$ and $x>0$,
\begin{equation}
\frac{\partial}{\partial x} \Phi(t,x)=-\lambda\left(\frac{\kappa(x+t/\lambda)}{x+t/\lambda}-\frac{\kappa(x)}{x}\right) \Phi(t,x),
\end{equation}
so $Z_t$ is a continuous random variable.
Let $f\in D$.  Clearly $T(t)f$ has compact support since $x-t/\lambda\vee Z_t\geq x-t/\lambda$.  
Since 
\[\left| \frac{f(\max(x+h-t/\lambda,y))-f(\max(x-t/\lambda,y))}{h}\right|\leq
  \| f' \|_\infty,\]
  and 
\begin{multline*}
\frac{f(\max(x+h-t/\lambda,y))-f(\max(x-t/\lambda,y))}{h}\\
\to \left\{\begin{array}{ll}f'(x-t/\lambda), & \text{ if } y \leq
  x-t/\lambda,\\ 0, & \text{ if } y > x-t/\lambda,\end{array}\right.
\end{multline*} 
as $h\to 0$, and since $\Phi(t,\cdot)$ is a continuous distribution function,
the dominated convergence theorem
gives $\frac{\partial}{\partial x} T(t)f(x)$ exists and is continuous.  In particular,
$\frac{\partial}{\partial x} T(t)f(x)\leq \int_{0}^{[x-t/\lambda]^+}
\|f'\|_\infty
\|\frac{\partial}{\partial x} \Phi(t,y)\|_\infty \,dy$, and so the derivative of
$T(t)f$ is bounded by a linear function.   Thus, $T(t)f\in D$ and
\cite{ethier2009markov} Proposition 1.3.3 implies that $D$ is a core for $A$.
\end{proof}

We are ready to prove Theorem \ref{thm:Markov}. Recall that for each $n\geq 1$, 
$\{Y_n(k),k=0,1,2,\ldots\}$ is the Markov chain in $[0,\infty)$ with transition
  function $\mu_n(x,\Gamma)=\probability{\max(x-I^{(n)}/n,M^{(n)}/n)\in
\Gamma}$ where $I^{(n)}$ is an exponential random variable with parameter
$\lambda^{(n)}$ independent of the random variable $M^{(n)}$, which itself is the
largest job in a busy period with service times equal in distribution to $V$
and interarrival times equal in distribution to $I^{(n)}$. Recall that
$Y_n(k)$ is equal in distribution to $\frac{1}{n}R_k^{(n)}$, and that we
define $X_n(t)=Y_n([nt])$. 

Let $T_n f(x)=\int f(y) \mu_n(x,dy)$ and let $A_n=n(T_n-I)$. The proof will
use the following formula for iterates of $T_n$.

\begin{lem}\label{l.TnRepn}
For all $n\ge 1$, $t\ge0$ and $x\ge0$,
\begin{equation*}
T_n^{[nt]}f(x)=\expectation{f\left(\max\left(x-\frac{1}{n}\sum_{j=1}^{[nt]}
I_j^{(n)},\frac{1}{n}\max_{k=1}^{[nt]}\left(M_k^{(n)}-\sum_{j=k}^{[nt]-1}
I_j^{(n)}\right)\right)\right)}
\end{equation*}
\end{lem}
\begin{proof}
If $[nt] = 2$ then 
\begin{equation*}
\begin{split}
T_n^{[nt]}f(x)&=T_n(T_n f)(x) = \int T_n f(y) \, \mu_n(x,dy)\\
&= \int \int f(z)\,\mu_n(y,dz) \, \mu_n(x,dy)\\
&= \expectation{\int f(z)\, \mu_n(\max(x-I_2^{(n)}/n,M_2^{(n)}/n),dz)}\\
&= \expectation{\expectation{
  f\left(\max\left(\max\left(x-I_2^{(n)}/n,M_2^{(n)}/n\right)-I_1^{(n)}/n,M_1^{(n)}\right)\right)}}\\
&= \expectation{
  f\left(\max\left(\max\left(x-I_2^{(n)}/n,M_2^{(n)}/n\right)-I_1^{(n)}/n,M_1^{(n)}\right)\right)},\\
\end{split}
\end{equation*}
where the iterated integral becomes the expectation of independent random
variables $I_1^{(n)},I_2^{(n)},M_1^{(n)},M_2^{(n)}$. Then
\begin{equation*}
\begin{split}
T_n^{[nt]}f(x) &= \expectation{
  f\left(\max\left(x-I_1^{(n)}/n-I_2^{(n)}/n,M_1^{(n)}/n-I_1^{(n)}/n,M_2^{(n)}/n\right)\right)}\\
&=\expectation{f\left(\max\left(x-\frac{1}{n}\sum_{j=1}^{2}
I_j^{(n)},\frac{1}{n}\max_{k=1}^{2}\left(M_k^{(n)}-\sum_{j=k}^{1}
I_j^{(n)}\right)\right)\right)}.
\end{split}
\end{equation*}
The general case follows by induction.
\end{proof}

{\it Proof of Theorem \ref{thm:Markov}}
Note that $T_n:\hat{C}([0,\infty))\to\hat{C}([0,\infty))$. So by
  \cite{ethier2009markov} Theorem 4.2.6 and Lemma \ref{lem:TtSemigroup} it
  suffices to show for each $f\in \hat C([0,\infty))$,
    $T_n^{[nt]}f$ converges uniformly to $T(t)f$ as $n\to\infty$, for each $t\geq0$.  To show this we show
    $A_nf\to Af$ for each $f\in D$ which by \cite{ethier2009markov} Theorem
    1.6.5 gives $T_n^{[nt]} f\to T(t)f$ uniformly on compact sets for each
    $f\in \hat C([0,\infty))$. Then to upgrade to uniform convergence we
      use the fact that $T_n^{[nt]}f$ (and $T(t)f$) are uniformly small for
      large $x$; in particular by Lemma \ref{l.TnRepn},
\begin{equation*}
\begin{split}
T_n^{[nt]}f(x)&=\expectation{f\left(\max\left(x-\frac{1}{n}\sum_{j=1}^{[nt]}
I_j^{(n)},\frac{1}{n}\max_{k=1}^{[nt]}\left(M_k^{(n)}-\sum_{j=k}^{[nt]-1}
I_j^{(n)}\right)\right)\right)}\\
&=\expectation{f\left(\max\left(x-\frac{1}{n}\sum_{j=1}^{[nt]}
I_j^{(n)},\frac{1}{n}R_{[nt]}^{(n)}\right)\right)}\\
&\leq \sup_{z\in [ [x-y]^+,\infty)}|f(z)| +\| f\|_\infty \probability{\frac{1}{n}\sum_{j=1}^{[nt]} I_j^{(n)} >y},
\end{split}
\end{equation*}
where the last inequality is true for any $y>0$. By setting $y>\lambda t$, the
limitting mean of the idle periods, and then choosing $x$ sufficiently
larger than $y$, both of the above terms can be made uniformly small in $n$
and $x$.

To show $A_nf\to Af$, fix $f\in D$ so the derivative of $f$ is bounded and let $a>0$ be such that $|f'(x)|\leq ax$. Write
\begin{equation}\label{eq:Anf}
\begin{split}
A_nf(x)&=\expectation{n(f(\max(x-I^{(n)}/n,M^{(n)}/n))-f(x))}\\
&=\int_{y=0}^\infty\int_{z=0}^\infty n(f(\max(x-y/n,z/n))-f(x))\lambda^{(n)}e^{-\lambda^{(n)}y} m^{(n)}(dz)\,dy\\
&=\int_{y=0}^\infty\int_{z=0}^{[nx-y]^+} n(f([x-y/n]^+)-f(x))\lambda^{(n)}e^{-\lambda^{(n)}y} m^{(n)}(dz)\,dy\\
&\qquad +\int_{y=0}^\infty\int_{z=[nx-y]^+}^{\infty} n(f(z/n)-f(x))\lambda^{(n)}e^{-\lambda^{(n)}y} m^{(n)}(dz)\,dy\\
&=\int_{y=0}^\infty m^{(n)}([nx-y]^+) n(f([x-y/n]^+)-f(x))\lambda^{(n)}e^{-\lambda^{(n)}y}\,dy\\
&\qquad +\int_{y=0}^\infty\int_{u=[x-y/n]^+}^{\infty} (f(u)-f(x))nm^{(n)}(ndu)\lambda^{(n)}e^{-\lambda^{(n)}y} \,dy.
\end{split}
\end{equation}
In the case $x=0$ the first integral is zero.
For $x>0$ we have $m^{(n)}([nx-y]^+)\uparrow 1$ for each $y$ as $n\to \infty$ since
$m^{(n)}\leq m^{(n+1)}$ by Lemma $\ref{lem:IncreasingM}$ and $m^{(\infty)}$ is
proper. Also as $n\to\infty$,
\begin{equation}
  n(f([x-y/n]^+)-f(x)))=-n\int_{z=[x-y/n]^+}^{x}f'(z)\,dz\to -yf'(x),
  \qquad y\ge 0,
\end{equation}
and $|n(f([x-y/n]^+)-f(x)))|\leq |y| \|f'\|_\infty$.  So, 
\[|m^{(n)}([nx-y]^+) n(f([x-y/n]^+)-f(x)))\lambda^{(n)}e^{-\lambda^{(n)}y}|\leq  |y|
  \|f'\|_\infty\lambda e^{-\lambda^{(1)} y}\]
 with $\lambda^{(n)}\uparrow\lambda$ positive and finite. Thus dominated convergence gives
\begin{equation}
\int_{y=0}^\infty m^{(n)}([nx-y]^+) n(f([x-y/n]^+)-f(x)))\lambda^{(n)}e^{-\lambda^{(n)}y}\,dy\to -\frac{f'(x)}{\lambda}
\end{equation}
for $x\geq0$, since $f'(0)=0$.

Now we may write the last line of $\eqref{eq:Anf}$ as
\begin{equation}\label{eq:smallBig}
\begin{split}
&\int_{y=0}^\infty\int_{u=[x-y/n]^+}^{x} (f(u)-f(x))nm^{(n)}(ndu)\lambda^{(n)}e^{-\lambda^{(n)}y} \,dy\\
&+\int_{y=0}^\infty\int_{u=x}^{\infty} (f(u)-f(x))nm^{(n)}(ndu)\lambda^{(n)}e^{-\lambda^{(n)}y} \,dy.
\end{split}
\end{equation}
If $x=0$ the first integral is zero. If $x>0$ then it goes to zero since
\begin{equation*}
\begin{split}
&\int_{u=[x-y/n]^+}^{x} (f(u)-f(x))nm^{(n)}(ndu)\\
&\leq y\| f'\| \left(m^{(n)}(nx) - m^{(n)}(n[x-y/n]^+)\right)\to 0
\end{split}
\end{equation*}
for each $y$ since $\|n(f(u)-f(x))\|\leq y\|f'\|$ for $u\in[{[x-y/n]^+},{x}]$.
The integral $\int_{y=0}^\infty\int_{u=[x-y/n]^+}^{x} (f(u)-f(x))nm^{(n)}(ndu)\lambda^{(n)}e^{-\lambda^{(n)}y} \,dy\to 0$ since the inner integral is bounded by $y\|f'\|$.
We may evaluate the outer integral of the second term in \eqref{eq:smallBig} immediately:
\begin{equation*}
\begin{split}
&\int_{y=0}^\infty\int_{u=x}^{\infty} (f(u)-f(x))nm^{(n)}(ndu)\lambda^{(n)}e^{-\lambda^{(n)}y} \,dy\\
&=\int_{u=x}^{\infty} (f(u)-f(x))nm^{(n)}(ndu)\\
&=\int_{u=x}^\infty \int_{z=x}^u f'(z)\,dz\, nm^{(n)}(ndu)\\
&=\int_{z=x}^\infty \int_{u=z}^\infty f'(z) nm^{(n)}(ndu)\,dz\\
&=\int_{z=x}^\infty f'(z) n\bar m^{(n)}(nz)\,dz.
\end{split}
\end{equation*}
Letting $K$ be the support of $f'(z)$, the integrand 
\[|f'(z) n\bar m^{(n)}(nz) |\leq azn\bar m^{(n)}(nz) 1_K(z)\to
  a\kappa(z)1_K(z)\]
  uniformly  because $n \bar m^{(n)}(nz)\to \kappa(z)/z$ uniformly by
  monotonicity of $\bar m^{(n)}(nz)$ and because $z\in K$ is bounded.  So
  bounded convergence gives 
  \[\int_{z=x}^\infty f'(z) n\bar m^{(n)}(nz)\,dz\to
    \int_{z=x}^\infty f'(z) \frac{\kappa(z)}{z}\,dz.\]

\hfill$\qed$

\bibliographystyle{acm}
\bibliography{GTZ}

\begin{thebibliography}{10}

\bibitem{BalleriniResnick}
{\sc Ballerini, R., and Resnick, S.~I.}
\newblock Records in the presence of a linear trend.
\newblock {\em Adv. in Appl. Probab. 19}, 4 (1987), 801--828.

\bibitem{billingsley1968convergence}
{\sc Billingsley, P.}
\newblock {\em Convergence of probability measures}.
\newblock Wiley Series in probability and Mathematical Statistics: Tracts on
  probability and statistics. Wiley, 1968.

\bibitem{boxmathesis}
{\sc Boxma, O.}
\newblock {\em Analysis of models for tandem queues}.
\newblock PhD thesis, University of Utrecht, Utrecht, 1977.

\bibitem{boxma1978longest}
{\sc Boxma, O.}
\newblock On the longest service time in a busy period of the
  \text{M\textbackslash G\textbackslash1 queue}.
\newblock {\em Stochastic Processes and their Applications 8}, 1 (1978),
  93--100.

\bibitem{boxma1979tandem}
{\sc Boxma, O.}
\newblock On a tandem queueing model with identical service times at both
  counters, i.
\newblock {\em Advances in Applied Probability\/} (1979), 616--643.

\bibitem{Boxma:2000}
{\sc Boxma, O.~J., and Deng, Q.}
\newblock Asymptotic behaviour of the tandem queueing system with identical
  service times at both queues.
\newblock {\em Mathematical Methods of Operations Research 52}, 2 (2000),
  307--323.

\bibitem{ethier2009markov}
{\sc Ethier, S.~N., and Kurtz, T.~G.}
\newblock {\em Markov processes: characterization and convergence}, vol.~282.
\newblock John Wiley \& Sons, 2009.

\bibitem{KarpelevichKreinin}
{\sc Karpelevich, F.~I., and Kre{\u\i}nin, A.~Y.}
\newblock {\em Heavy traffic limits for multiphase queues}, vol.~137 of {\em
  Translations of Mathematical Monographs}.
\newblock American Mathematical Society, Providence, RI, 1994.
\newblock Translated from the Russian manuscript by Kre{\u\i}nin and A.
  Vainstein.

\bibitem{kk96}
{\sc Karpelevitch, F.~I., and Kreinin, A.~Y.}
\newblock Asymptotic analysis of queueing systems with identical service.
\newblock {\em J. Appl. Probab. 33}, 1 (1996), 267--281.

\bibitem{RegVar}
{\sc N.~H.~Bingham, C. M.~Goldie, J. L.~T.}
\newblock {\em Regular Variation}.
\newblock Cambridge University Press, 1987.

\bibitem{Resnick}
{\sc Resnick, S.~I.}
\newblock {\em Heavy-tail phenomena}.
\newblock Springer Series in Operations Research and Financial Engineering.
  Springer, New York, 2007.
\newblock Probabilistic and statistical modeling.

\end{thebibliography}

\vspace{2ex}
\begin{minipage}[t]{2in}
\footnotesize {\sc Department of Mathematics\\
University of Virginia\\
Charlottesville, VA 22904\\
E-mail: gromoll@virginia.edu\\
E-mail: bat5ct@virginia.edu}
\end{minipage}
\hspace{5ex}
\begin{minipage}[t]{4in}
\footnotesize {\sc Centrum Wiskunde \& Informatica\\
P.O. Box 94079\\
1090 GB Amsterdam, Netherlands\\
E-mail: Bert.Zwart@cwi.nl}

\end{minipage}

\end{document}